\documentclass[a4paper,12pt]{amsart}

\usepackage{amssymb}
\usepackage{amsmath}
\usepackage{amsthm}

\usepackage[colorlinks]{hyperref}
\usepackage{hyperref}
\renewcommand\eqref[1]{(\ref{#1})} %Need with hyperref

\allowdisplaybreaks \numberwithin{equation}{section}

\theoremstyle{plain}
\newtheorem{theorem}{Theorem}[section]
\newtheorem{prop}[theorem]{Proposition}

\newtheorem{lemma}[theorem]{Lemma}

\theoremstyle{definition}
\newtheorem{defn}[theorem]{Definition}
\newtheorem{rem}[theorem]{Remark}

\begin{document}

\title[Pseudo-parabolic equations associated with the Jacobi operator]{Inverse source problem for the pseudo-parabolic equation associated with the Jacobi operator}

\author[B. Bekbolat ]{Bayan Bekbolat}
\address{
  Bayan Bekbolat :
  \endgraf
  Al-Farabi Kazakh National University
  \endgraf
  Almaty, Kazakhstan
  \endgraf
  and
  \endgraf 
  Department of Mathematics: Analysis, Logic and Discrete Mathematics
  \endgraf
  Ghent University, Belgium
  \endgraf
  and
  \endgraf   
  Institute of Mathematics and Mathematical Modeling
  \endgraf
  Almaty, Kazakhstan
  \endgraf  
  and
  \endgraf 
  Suleyman Demirel University
  \endgraf
  Kaskelen, Kazakhstan
  \endgraf
  {\it E-mail address:} {\rm bekbolat@math.kz}
  }

\author[N. Tokmagambetov ]{Niyaz Tokmagambetov }
\address{
Niyaz Tokmagambetov:
 \endgraf 
  Centre de Recerca Matem\'atica
  \endgraf
  Campus de Bellaterra, Edifici C, 08193 Barcelona, Spain
  \endgraf
  and
  \endgraf 
   Institute of Mathematics and Mathematical Modeling
  \endgraf
  Almaty, Kazakhstan
  \endgraf
  {\it E-mail address} {\rm tokmagambetov@crm.cat; tokmagambetov@math.kz}
  }

\date{\today}

\thanks{This research was funded by the Science Committee of the Ministry of Science and Higher Education of the 
Republic of Kazakhstan (Grant No. AP14972634).}

\keywords{Jacobi operator, Jacobi transform,  time-fractional
pseudo-parabolic equation, inverse source problem.}
\subjclass[2020]{Primary 35R30; Secondary 35R11, 35C15.}

\maketitle

\begin{abstract}
In this paper we investigate direct and inverse problems for time-fractional pseudo-parabolic equations associated with the Jacobi operator. The existence and uniqueness of the solutions are proved. Also, the stability result of the inverse source problem (ISP) is established. 
\end{abstract}

%\tableofcontents

\section{Introduction}

The main object of this paper is the following non-homogeneous time-fractional pseudo-parabolic equation on the domain $D=\{(t,x) : 0<t<T<\infty, \, x\in\mathbb{R}^+=(0,\infty)\}$ 
\begin{equation*} 
\mathbb{D}_{0^+,t}^\gamma\left( u(t,x)-a\Delta_{\alpha,\beta}u(t,x)\right)-\Delta_{\alpha,\beta}u(t,x)+mu(t,x)=f(x),
\end{equation*}
where $0<\gamma\leq1$, with non--negative constants $m$ and $a$, and with the initial condition
\begin{equation*}
u(0,x)=\phi(x),\quad x\in\mathbb{R}^+,
\end{equation*}
where $\mathbb{D}_{0^+,t}^\gamma$ is given by
\begin{equation*}
\mathbb{D}_{0^+,t}^\gamma =
\begin{cases}
     \mathcal{D}_{0^+,t}^\gamma, \quad 0<\gamma<1,\\
     \frac{d}{dt}, \quad \gamma=1,
\end{cases}
\end{equation*}
$\mathcal{D}_{0^+,t}^\gamma$ is the left-sided Caputo fractional derivative and $\Delta_{\alpha,\beta}$ is the Jacobi operator given by the expression 
\begin{equation}\label{EQ: 1.1}
\Delta_{\alpha,\beta}=A_{\alpha,\beta}^{-1}(x)\frac{d}{dx}\Bigl(A_{\alpha,\beta}(x)\frac{d}{dx}\Bigr),\quad x\in(0,\infty).
\end{equation}\\
Here, we denote by $A_{\alpha,\beta}(x)=2^{2\rho}(\sinh(x))^{2\alpha+1}(\cosh(x))^{2\beta+1}, \rho=\alpha+\beta+1$, with $\alpha\geq\beta\geq-\frac{1}{2}$.

In our studies we would be questioned about the well--posedness of the direct problem and the stability of the inverse source problem with the additional information --  over-determination condition
\begin{equation*}
u(T,x)=\psi(x),\quad x\in\mathbb{R}^+.
\end{equation*}
For the ISP we will restore the pair $(u, f)$ under some conditions on the function $\psi$.

One of the first mathematicians who studied the ISP was Rundell \cite{Run80} in 1980s. He considered the  evolution type equation
\begin{equation} \label{DiffEqu}
    \frac{du}{dt}+Au=f
\end{equation}
in a Banach space $X$, where $A$ is linear operator in $X$ and $f$ is a constant vector in $X$, with conditions
\begin{equation*}
    u(0)=u_0, \quad \text{and} \quad u(T)=u_1.
\end{equation*}
Using semigroups of operators Rundell proved a general theorem about the existence of a unique solution pair $(u(t),f)$ of the problem, which then was applied to equations of parabolic and pseudo-parabolic types. When the non-homogeneous term is represented in the form $f(t)=\Phi(t)f$, where $\Phi(t)$ is known operator and the element $f$ is unknown, and $A$ is a closed linear operator from $L_p(\Omega)$ into $L_p(\Omega)$, several ISPs for the equation \eqref{DiffEqu} were studied by A.I. Prilepko and I.V. Tikhonov \cite{PT92} in 1992. They applied obtained results to the transport equation. In the general case, where the unknown source depends on time, under a sufficient condition, ISPs for the equation \eqref{DiffEqu} with the linear elliptic partial differential operator $A$ of order $2m$ with the bounded measurable coefficients such that
\begin{equation*}
    (A\varphi,\varphi)\geq \|\varphi\|^2 
\end{equation*}
for all $\varphi\in H^{2m}(\Omega)\cap H_0^{m}(\Omega)$, $\mu=constant>0$ was investigated by I. Bushuyev \cite{B95} in 1995. 

Nonetheless, there is no general closed theory for abstract case of $F(x,t)$. Known results deal with separated source terms. In 2002 I.V. Tikhonov and Yu.S. Eidelman \cite{TE02} considered ISPs for the generalization of the equation \eqref{DiffEqu} of the form
\begin{equation*}
    \frac{d^Nu(t)}{dt^N} = Au(t) + p, \quad 0<t<T,
\end{equation*}
for some positive integer $N\geq1$ and some real number $T>0$ with an unknown parameter $p$ and a closed linear operator $A$ in the Banach space under the Cauchy conditions and "over-determination condition" $u(T)=u_N$ (also in the Banach space). 

For the Laplace operator $(-\Delta)$ which is one of the most interesting examples in Physics, M. Choulli and M. Yamamoto in \cite{CY04} established the uniqueness and conditional stability in determining a heat source term from boundary measurements with $f=\sigma(t)\varphi(x)$, where $\sigma(t)$ is known.

Asymptotic behaviour of the solution of the inverse source problem for the pseudo-parabolic equation
\begin{equation*}
    (u(x,t) - \Delta u(x,t))_t - \Delta u(x,t) + \alpha u(x,t) = f(t)g(x,t), \quad Q_\infty = \Omega\times(0,\infty)
\end{equation*}
with a integral over-determination condition was studied by M. Yaman and \"{O}. F. G\"{o}z\"{u}kızıl in \cite{YG03} in 2004. 

Fractional derivatives and fractional partial differential equations have received great attention both in analysis and application, which are used in modeling several phenomena in different areas of science such as biology, physics, and chemistry, so the fractional computation is increasingly attracted to mathematicians in the last several decades. ISP for the time fractional parabolic equation
\begin{equation*}
    {}^c D_t^\alpha u(x,t) = r^\alpha (Lu)(x,t) + f(x)h(x,t), \quad x\in\Omega,\,\ t\in(0,T),\,\ 0<\alpha<1,
\end{equation*}
where ${}^c D_t^\alpha$ is the Caputo derivative defined by
\begin{equation*}
    {}^c D_t^\alpha g(t) = \frac{1}{\Gamma(1-\alpha)} \int_0^t (t-\tau)^{-\alpha} \frac{d}{d\tau}g(\tau)d\tau
\end{equation*}
and $L$ is a symmetric uniformly elliptic operator was considered by K. Sakamoto and M. Yamamoto in \cite{SY11} in 2011. The authors proved that the inverse problem is
well-posed in the Hadamard sense except for a discrete set of values of diffusion constants using final overdetermining data. Blow-up solution and stability to ISP for the pseudo-parabolic equation
\begin{equation*}
    u_t - a\Delta u_t - \Delta u + \sum_{i=1}^n b_i u_{x_i} - |u|^p u = f(t)g(t), \quad x\in\Omega, t>0
\end{equation*}
with the integral overdetermination condition was studied by Metin Yaman in \cite{Yam12} in 2012. ISP for the equation \eqref{DiffEqu} considered by M.M. Slodi\u{c}ka in \cite{S13} in 2013, when $A$ is a linear differential operator of second-order, strongly elliptic, and the right-hand side $f$ is assumed to be separable in both variables $x$ and $t$, i.e. $f(x,t)=g(x)h(t)$ (in this case $h(t)$ is unknown). ISP for a semilinear time-fractional diffusion equation of second order in a bounded domain in $\mathbb{R}^d$
\begin{equation*}
    (g_{1-\beta}*\partial_t u(x))(t) + L(x,t)u(x,t) = h(t)f(x) + \int_0^t F(x,s,u(x,s))ds
\end{equation*}
with a linear second order differential operator $L(x,t)$ in the divergence form with space and time dependent coefficients was studied by M. Slodi\u{c}ka and M. \u{S}i\u{s}kova in \cite{SS16} in 2016. Authors showed the existence, uniqueness and
regularity of a weak solution $(u,h)$ (\cite[Theorem 2.1, p. 1658]{SS16}). One of the recent papers for inverse source problems for pseudo-parabolic equations with fractional derivatives is \cite{RSTT21} (in 2021). In \cite{RSTT21}, authors have considered solvability of an inverse source problem for the pseudo-parabolic equation with the Caputo fractional derivative $\mathcal{D}_t^\alpha$ of order $0<\alpha\leq1$
\begin{equation*}
    \mathcal{D}_t^\alpha(u(t)+\mathcal{L}u(t))+\mathcal{M}u(t)=f(t) \quad \text{in} \quad \mathcal{H},
\end{equation*}
\begin{equation*}
    u(0)=\phi\in\mathcal{H}, \quad u(T)=\psi\in\mathcal{H},
\end{equation*}
where $\mathcal{H}$ be a separable Hilbert space and $\mathcal{L}$, $\mathcal{M}$ be operators with the corresponding discrete spectra on $\mathcal{H}$. The authors obtained well-posedness results.

A number of articles address the solvability of the inverse problems for the diffusion and sub-diffusion equations (\cite{CNYY09, JR15, KS10, KST17, OS12a, OS12b,  RTT19}) and fractional diffusion equations (\cite{SSB19, TT17, WYH13}).

The semigroups $(H_t^{(\alpha,\beta)})_{t\geq0}$ (the solution of the heat equation associated with the Jacobi-Dunkl operator $\Lambda_{\alpha,\beta}^2$ ) generate a new family of Markov processes on the real line. On some Riemannian symmetric spaces this process is the radial part of the Brownian motion for particular values of $(\alpha,\beta)$ \cite{CGM06}. 

However, the ISP for the pseudo-parabolic equations generated by the Jacobi operator $\Delta_{\alpha,\beta}$ \eqref{EQ: 1.1} have not been still considered. So, our goal is to consider the ISP for the pseudo-parabolic equation with this special operator. Harmonic analysis associated with the operator $\Delta_{\alpha,\beta}$ has been studied by M. Flensted-Jensen and T. H. Koornwinder \cite{FJ72, FJK73, FJK79, Koo75}. The spectral decomposition of the Jacobi operator was considered by M. Flensted-Jensen in 1972 \cite{FJ72}. There were obtained a generalization of the classical Paley-Wiener Theorem and a generalized Fourier transform $\mathcal{F}_{\alpha,\beta}$, is called Jacobi-Fourier transform. Eigenfunctions $\varphi_\lambda ^{\alpha, \beta}(x)$ of the operator Jacobi is called the Jacobi function, which is hypergeometric function. The pseudo-differential operators (see \cite{SD98}) and Sobolev type spaces $G_{\alpha,\beta}^{s,p}$ (see \cite{SD00}) associated with the Jacobi operator was studied by N. Ben Salem and A. Dachraoui. In \cite{SD98}, authors proved that a pseudo-differential operator associated with a symbol in $S_0^m$ is a continuous linear mapping from some subspace of the Schwartz space into itself.

Our main result reads as follows.

\begin{theorem} Let $0<\gamma\leq1$. Assume that $\psi,\phi\in\mathcal{H}$. Then the pair $(u,f)$ is a unique solution of the ISP, which are functions $u\in C^\gamma([0,T],L^{2}(\mu))\cap C([0,T],\mathcal{H}), f\in L^{2}(\mu)$ can be represented by the formulas
\begin{multline*}
    u(t,x)=\int_0^\infty\int_0^\infty\frac{1-\mathbb{E}_{\gamma,1}\left(-\frac{\lambda^{2}+\rho^2+m}{1+a(\lambda^{2}+\rho^2)}t^\gamma\right)}{1-\mathbb{E}_{\gamma,1}\left(-\frac{\lambda^{2}+\rho^2+m}{1+a(\lambda^{2}+\rho^2)}T^\gamma\right)}\psi(y)\varphi_{\lambda}^{\alpha,\beta}(y)\varphi_{\lambda}^{\alpha,\beta}(x)d\mu_{\alpha,\beta}(y)d\nu_{\alpha,\beta}(\lambda)\\
    -\int_0^\infty\int_0^\infty\frac{\mathbb{E}_{\gamma,1}\left(-\frac{\lambda^{2}+\rho^2+m}{1+a(\lambda^{2}+\rho^2)}T^\gamma\right)-\mathbb{E}_{\gamma,1}\left(-\frac{\lambda^{2}+\rho^2+m}{1+a(\lambda^{2}+\rho^2)}t^\gamma\right)}{1-\mathbb{E}_{\gamma,1}\left(-\frac{\lambda^{2}+\rho^2+m}{1+a(\lambda^{2}+\rho^2)}T^\gamma\right)}\\
    \times\phi(y)\varphi_{\lambda}^{\alpha,\beta}(y)\varphi_{\lambda}^{\alpha,\beta}(x)d\mu_{\alpha,\beta}(y)d\nu_{\alpha,\beta}(\lambda)
\end{multline*}
and
\begin{multline*}
    f(x)=\int_0^\infty\int_0^\infty(\lambda^{2}+\rho^2+m)\frac{\psi(y)-\phi(y)\mathbb{E}_{\gamma,1}\left(-\frac{\lambda^{2}+\rho^2+m}{1+a(\lambda^{2}+\rho^2)}T^\gamma\right)}{1-\mathbb{E}_{\gamma,1}\left(-\frac{\lambda^{2}+\rho^2+m}{1+a(\lambda^{2}+\rho^2)}T^\gamma\right)}\\
    \times\varphi_{\lambda}^{\alpha,\beta}(y)\varphi_{\lambda}^{\alpha,\beta}(x)d\mu_{\alpha,\beta}(y)d\nu_{\alpha,\beta}(\lambda).
\end{multline*}
\end{theorem}

The contents of this paper as follows. In Section 2, we collect some results about harmonic analysis associated with the Jacobi operator on $\mathbb{R}^+$ and here we introduce the Sobolev type space $\mathcal{H}$, also given some necessary information about fractional derivative. In Section 3, we prove Theorem \ref{Theorem1} for the direct problem. In Section 4, we prove our main Theorem \ref{Theorem2} about solvability of the inverse source problem associated with the Jacobi operator on $\mathbb{R}^+$, also shown stability analysis and example for the inverse source problem.

\section{Preliminaries}

\subsection{Jacobi analysis}

The singular second order differential equation (\cite{FJ72})
\begin{equation*}
    \Delta_{\alpha, \beta}\varphi_\lambda ^{\alpha, \beta}(x)+(\lambda ^{2}+\rho ^{2})\varphi_\lambda ^{\alpha, \beta}(x)=0 \quad \text{on} \quad (0,\infty)
\end{equation*}
with initial conditions
\begin{equation*}
    \varphi_\lambda ^{\alpha, \beta}(0)=1, \quad \frac{d}{dt}\varphi_\lambda ^{\alpha, \beta}(0)=0
\end{equation*}
has a unique solution, given by the expression
\begin{equation}
\label{Jacobi Function}
    \varphi_\lambda ^{\alpha, \beta}(x) = {}_2F_1\Bigl( \frac{1}{2}(\rho +i\lambda), \frac{1}{2}(\rho -i\lambda);\alpha +1;-\sinh^{2}x \Bigr),
\end{equation}
where ${}_2F_1$ is the Gauss hypergeometric function. The function $\varphi_\lambda ^{\alpha, \beta}$ \eqref{Jacobi Function} is called the Jacobi function and analytic for $x\in[0,\infty)$ and
\begin{equation*}
    \varphi_\lambda ^{\alpha, \beta}(x)=\varphi_{-\lambda} ^{\alpha, \beta}(x) \quad \text{and} \quad \overline{\varphi_{\lambda}^{\alpha, \beta}(x)}=\varphi_{\overline{\lambda}}^{\alpha, \beta}(x).
\end{equation*}
In particularly, we have
\begin{equation*}
    \varphi_\lambda ^{-\frac{1}{2}, -\frac{1}{2}}(x)=\cos(\lambda x).
\end{equation*}
\begin{rem} (\cite[Proposition 1, p. 144]{FJ72})
For each fixed $x\in(0,\infty)$, as a function of $\lambda$, $\varphi_\lambda ^{\alpha, \beta}$ is an entire function.  
\end{rem}

Properties of the Jacobi function:\\
1. For all $\lambda \in \mathbb{C}$ and $x\in[0,\infty)$, we have (\cite[Lemma 11, p. 153]{FJ72})\\
i) $|\varphi_\lambda ^{\alpha, \beta}(x)|\leq \varphi_{i Im\lambda}^{\alpha, \beta}(x)$,\\
ii) If $|Im\lambda|\geq\rho$ then $|\varphi_\lambda ^{\alpha, \beta}(x)|\leq e^{(|Im\lambda|-\rho)x}$,\\
iii) If $|Im\lambda|\leq\rho$ then $|\varphi_\lambda ^{\alpha, \beta}(x)|\leq 1$.\\
2. For all $n\in\mathbb{Z}^+$ there exists $K_{n}>0$ such that (\cite[Theorem 2, p. 145]{FJ72})
\begin{equation*}
\left|\frac{d^{n}}{dx^{n}}\varphi_\lambda ^{\alpha, \beta}(x)\right|\leq K_{n}(1+x)(1+|\lambda|)^{n}e^{(|Im\lambda|-\rho)x}
\end{equation*}
and
\begin{equation*}
\left|\frac{d^{n}}{d\lambda^{n}}\varphi_\lambda ^{\alpha, \beta}(x)\right|\leq K_{n}(1+x)^{n+1}e^{(|Im\lambda|-\rho)x}
\end{equation*}
for all $\lambda\in\mathbb{C}$, $x\in[0,\infty)$.

Let us introduce the following functions spaces (\cite[p. 146-147]{FJ72}, \cite[p. 368]{SD98}).

Let $\mathcal{S}_e(\mathbb{R})$ be the space of even, rapidly decreasing, and $C^\infty$-functions on $\mathbb{R}$, equipped with usual Schwartz topology, and $\mathcal{S}_e^{r}(\mathbb{R})=\{(\cosh x)^{\frac{-2\rho}{r}}\mathcal{S}_e(\mathbb{R})\}$, $0<r\leq2$ be the space with the topology defined by the semi-norms
\begin{equation*}
    N_{n,k}(f)=\sup_{x\geq0}(\cosh x)^{\frac{2\rho}{r}}(1+x)^{n}\left|\frac{d^k}{dx^k}f(x)\right|.
\end{equation*}
Clearly $\mathcal{S}_e^{r}(\mathbb{R})$ is invariant under $\Delta_{\alpha, \beta}$ and the semi-norms defined by
\begin{equation*}
    N_{n,k}(f)=\sup_{x\geq0}(\cosh x)^{\frac{2\rho}{r}}(1+x)^{n}|\Delta_{\alpha,\beta}^k f(x)|
\end{equation*}
are continuous on $\mathcal{S}_e^{r}(\mathbb{R})$.

Let $L^{p}(\mathbb{R}^+,\mu_{\alpha,\beta}), 1\leq p<\infty$ be the space of measurable functions $f$ on $\mathbb{R}^+$ such that
\begin{equation*}
    \|f\|_{p,\mu}^p=\int_0^{\infty}|f(x)|^{p}d\mu_{\alpha,\beta}(x)<\infty,
\end{equation*}
where $d\mu_{\alpha,\beta}(x)=(2\pi)^{-\frac{1}{2}}2^{2\rho}(\sinh x)^{2\alpha+1}(\cosh x)^{2\beta+1}dx$ or $d\mu_{\alpha,\beta}(x)=(2\pi)^{-\frac{1}{2}}A_{\alpha,\beta}(x)dx$. 
\begin{rem}\cite[p. 146]{FJ72}
Notice that $\mathcal{S}_e^{r}(\mathbb{R})\subset L^{r}(\mathbb{R}^+,\mu_{\alpha,\beta})$ for all $0<r\leq2$ and if $r\leq s$ then $\mathcal{S}_e^{r}(\mathbb{R})\subseteq\mathcal{S}_e^{s}(\mathbb{R})\subset L^{2}(\mathbb{R}^+,\mu_{\alpha,\beta})$.
\end{rem}
Let
$L^{p}(\mathbb{R}^+,\nu_{\alpha,\beta}), 1\leq p<\infty$ be the space of measurable functions $g$ on $\mathbb{R}^+$ such that
\begin{equation*}
    \|f\|_{p,\nu}^p=\int_0^{\infty}|g(\lambda)|^{p}d\nu_{\alpha,\beta}(\lambda)<\infty,
\end{equation*}
where $d\nu_{\alpha,\beta}(\lambda)=(2\pi)^{-\frac{1}{2}}|c_{\alpha, \beta}(\lambda)|^{-2}d\lambda.$ Here, $c_{\alpha, \beta}(\lambda)$ is the Harish–Chandra function, given by
\begin{equation*}
    c_{\alpha, \beta}(\lambda) = \frac{2^{\rho-i\lambda}\Gamma(i\lambda)\Gamma(\alpha+1)}{\Gamma(\frac{\rho+i\lambda}{2})\Gamma(\frac{\alpha-\beta+1+i\lambda}{2})}.
\end{equation*}

For short, we use notations $L^{p}(\mu)$ and $L^{p}(\nu)$ instead $L^{p}(\mathbb{R}^+,\mu_{\alpha,\beta})$ and $L^{p}(\mathbb{R}^+,\nu_{\alpha,\beta})$, respectively.

For $f\in L^{1}(\mu)$ the Fourier-Jacobi transform $\mathcal{F}_{\alpha,\beta}$ of $f$ is defined by (\cite[Proposition 3, p. 146]{FJ72}, \cite[Definition 1.1, p. 369]{SD98})
\begin{equation}
\label{Jacobi Transform}
\widehat{f}(\lambda)=(\mathcal{F}_{\alpha,\beta}f)(\lambda)=\int_0^\infty f(x)\varphi_{\lambda}^{\alpha,\beta}(x) d\mu_{\alpha,\beta}(x)
\end{equation}
and for $g\in L^{1}(\nu)$ the inverse Fourier-Jacobi transform $\mathcal{F}_{\alpha,\beta}^{-1}$ is given by
\begin{equation}
\label{Invers Jacobi Transform}
\Bigl(\mathcal{F}_{\alpha,\beta}^{-1}g\Bigr)(x)=\int_0^\infty g(\lambda)\varphi_{\lambda}^{\alpha,\beta}(x) d\nu_{\alpha,\beta}(\lambda),
\end{equation}
where $\varphi_{\lambda}^{\alpha,\beta}$ is the Jacobi functions \eqref{Jacobi Function}.
\begin{prop} (\cite[p. 145-146]{FJ72})
The operator in $L^2(\mu)$ defined by $\Delta_{\alpha,\beta}$ with domain 
\begin{equation*}
    D_{\alpha,\beta}^0 = \{ u\in L^2(\mu) \,\, : \,\, u \quad \text{and} \quad u' \quad \text{are absolutely continuous and} \quad \Delta_{\alpha,\beta} u\in L^2(\mu) \}
\end{equation*}
can be restricted to a domain $D_{\alpha,\beta}$, such that $\Delta_{\alpha,\beta}$ becomes self-adjoint. $\Delta_{\alpha,\beta}$ contains at least functions in $D_{\alpha,\beta}^0$ which are differentiable at zero. $\Delta_{\alpha,\beta}$ has limit-point at $\infty$; and at zero there is limit-point if $2\alpha+1\geq3$, and limit-circle if $2\alpha+1<3$. In this last case $D_{\alpha,\beta}\neq D_{\alpha,\beta}^0$ and choosing $\lambda_1\in\mathbb{C}$ with $Im\lambda_1^2>0$ we can define
\begin{equation*}
    D_{\alpha,\beta} = \{ u\in D_{\alpha,\beta}^0 \,\, : \,\, \lim_{x\rightarrow0} (A_{\alpha,\beta}(x)\cdot(\varphi_{\lambda_1}^{\alpha,\beta}(x)\overline{u'(x)}-\left(\frac{d}{dx}\varphi_{\lambda_1}^{\alpha,\beta}(x)\right)\overline{u(x)}))=0 \}.
\end{equation*}
\end{prop}
\begin{prop}(\cite[Proposition 3, p. 146]{FJ72}) \label{Norm-preserving map}
For $f\in L^2(\mu)$ and $\lambda\in\mathbb{R}^+$ define $\widehat{f}$ the integral converging in $L^2(\nu)$. $f\rightarrow\widehat{f}$ is a linear, normpreserving map of $L^2(\mu)$ onto $L^2(\nu)$, the inverse given by
\begin{equation*}
    f(x) = \int_0^\infty g(\lambda)\varphi_{\lambda}^{\alpha,\beta}(x) d\nu_{\alpha,\beta}(\lambda)
\end{equation*}
the integral converging in $L^2(\mu)$. A function $f\in L^2(\mu)$ belongs to $D_{\alpha,\beta}$ if and only if $(\lambda^2+\rho^2)\widehat{f}(\lambda)\in L^2(\nu)$ and in that case
\begin{equation*}\label{EQ: 1.19}
    \widehat{\Delta_{\alpha,\beta}f}(\lambda) = - (\lambda^2+\rho^2)\widehat{f}(\lambda).
\end{equation*}
\end{prop}
In particularly, we have for Plancherel's identity 
\begin{equation}
\label{Plancherel's identity}
\|\widehat{f}\|_{2,\nu}=\|f\|_{2,\mu}.
\end{equation} 
\begin{rem} For $\alpha = \beta =-\frac{1}{2}$, we have the Fourier-cosine transform
\begin{equation*}
    \widehat{f}_{c}(\lambda)=(\mathcal{F}_{c}f)(\lambda)=\frac{1}{\sqrt{2\pi}}\int_0^{\infty}\cos(\lambda x)f(x)dx,
\end{equation*}
and the inverse Fourier-cosine transform is defined by
\begin{equation*}
    \Bigl(\mathcal{F}_{c}^{-1}g\Bigr)(x)=\frac{4}{\sqrt{2\pi}}\int_0^{\infty}\cos(\lambda x)g(\lambda)d\lambda.
\end{equation*}
\end{rem}
\begin{defn}We define the space
\begin{equation*}
    \mathcal{H} := \{ u\in L^2(\mu) \,\, : \,\, (\cdot^2+\rho^2)\widehat{u}\in L^2(\nu) \}
\end{equation*}
with norm
\begin{equation*}
    \|u\|_{\mathcal{H}}^2 := \int_0^\infty |(\lambda^2+\rho^2)\widehat{u}(\lambda)|^2d\nu_{\alpha,\beta}(\lambda).
\end{equation*}
\end{defn}

\subsection{Fractional differentiation operators}

In this subsection, we introduce fractional differentiation operators and other conceptions.

\begin{defn} \cite[p. 69]{KST06}
Let $[a, b]$ $(-\infty<a<b<\infty)$ be a finite interval on the real axis $\mathbb{R}$. The left and right Riemann-Liouville fractional integrals $I_{a^+}^{\gamma}$ and $I_{b^-}^{\gamma}$ of order $\gamma\in\mathbb{R}$ $(\gamma>0)$ are defined by
\begin{equation*}
    I_{a^+}^{\gamma}[f](t) := \frac{1}{\Gamma(\gamma)} \int_a^t (t-s)^{\gamma-1}f(s)ds, \quad t\in(a,b],
\end{equation*}
and
\begin{equation*}
    I_{b^-}^{\gamma}[f](t) := \frac{1}{\Gamma(\gamma)} \int_t^b (t-s)^{\gamma-1}f(s)ds, \quad t\in[a,b),
\end{equation*}
respectively. Here $\Gamma$ denotes the Euler gamma function.
\end{defn}
\begin{defn} \cite[p. 70]{KST06}
The left and right Riemann-Liouville fractional derivatives $D_{a^+}^{\gamma}$ and $D_{b^-}^{\gamma}$ of order $\gamma\in\mathbb{R}$ $(0<\gamma<1)$ are given by
\begin{equation*}
    D_{a^+}^{\gamma}[f](t) := \frac{d}{dt} I_{a^+}^{1-\gamma}[f](t), \quad \forall t\in(a,b],
\end{equation*}
and
\begin{equation*}
    D_{b^-}^{\gamma}[f](t) := -\frac{d}{dt} I_{b^-}^{1-\gamma}[f](t), \quad \forall t\in[a,b),
\end{equation*}
respectively. 
\end{defn}
\begin{defn} \cite[p. 91]{KST06}
The left and right Caputo fractional derivatives $D_{a^+}^{\gamma}$ and $D_{b^-}^{\gamma}$ of order $\gamma\in\mathbb{R}$ $(0<\gamma<1)$ are defined by
\begin{equation*}
    \mathcal{D}_{a^+}^{\gamma}[f](t) := D_{a^+}^{\gamma}[f(t)-f(a)], \quad t\in(a,b],
\end{equation*}
and
\begin{equation*}
    \mathcal{D}_{b^-}^{\gamma}[f](t) := D_{b^-}^{\gamma}[f(t)-f(b)], \quad t\in[a,b),
\end{equation*}
respectively. 
\end{defn}
\begin{defn} \cite[p. 18, Definition 3]{CF18} \label{DefinitionSpace}
Let $X$ be a Banach space. We say that $u\in C^\gamma([0,T],X)$ if $u\in C([0,T],X)$ and $\mathcal{D}_t^\gamma u\in C([0,T],X)$.
\end{defn}

The classical Mittag-Leffler function  $\mathbb{E}_{\gamma,1}(t)$ and the Mittag-Leffler type function $\mathbb{E}_{\gamma,\gamma}(t)$ are given by the expressions
\begin{equation*}
    \mathbb{E}_{\gamma,1}(t) := \sum_{k=0}^\infty \frac{t^k}{\Gamma(\gamma k+1)} \quad \mathbb{E}_{\gamma,\gamma}(t) := \sum_{k=0}^\infty \frac{t^k}{\Gamma(\gamma k+\gamma)}.
\end{equation*}
In the case $\gamma=1$, we obtain $\mathbb{E}_{1,1}(t)=e^t$. For more information about the classical Mittag-Leffler function $\mathbb{E}_{\gamma,1}(t)$ and the Mittag-Leffler type function $\mathbb{E}_{\gamma,\gamma}(t)$ see e.g. \cite[p. 40 and p. 42]{KST06}.

In \cite[Theorem 4, p. 21]{Sim14} the following estimate for the Mittag-Leffler function is proved, when $0<\gamma<1$ (not true for $\gamma\geq1$)
\begin{equation*}
    \frac{1}{1+\Gamma(1-\gamma)t}\leq\mathbb{E}_{\gamma,1}(-t)\leq\frac{1}{1+\Gamma(1+\gamma)^{-1}t}, \quad t>0.
\end{equation*}
Then it follows 
\begin{equation} \label{InEq}
    0< \mathbb{E}_{\gamma,1}(-t) <1, \quad t>0.
\end{equation}

\begin{prop}\cite{Pod99}
If $0<\gamma<2$, $\beta$ is an arbitrary real number, $\mu$ is such that $\pi\gamma/2<\mu<\min\{\pi,\pi\gamma\}$, then there exists positive constant $C$, such that we have
\begin{equation*}
    \left| \mathbb{E}_{\gamma,\beta}(z) \right| \leq \frac{C}{1+|z|}
\end{equation*}
for all $\mu\leq|\arg(z)|\leq\pi$.
\end{prop}

\begin{lemma} \label{lemma}
Assume that $0<t<T$, $0<\gamma\leq1$ and $\lambda\in\mathbb{R}^+$. Then 
\begin{equation} \label{InEq1}
    0 < \frac{1-\mathbb{E}_{\gamma,1}\left(-\lambda t^\gamma\right)}{1-\mathbb{E}_{\gamma,1}\left(-\lambda T^\gamma\right)} < 1
\end{equation}
and 
\begin{equation} \label{InEq2}
    -1 < \frac{\mathbb{E}_{\gamma,1}\left(-\lambda T^\gamma\right)-\mathbb{E}_{\gamma,1}\left(-\lambda t^\gamma\right)}{1-\mathbb{E}_{\gamma,1}\left(-\lambda T^\gamma\right)} < 0
\end{equation}
inequalities hold.
\end{lemma}
\begin{proof}
Using property \eqref{InEq} we have
\begin{equation*}
    0 < 1-\mathbb{E}_{\gamma,1}\left(-\lambda T^\gamma\right) < 1
\end{equation*}
or
\begin{equation} \label{InEq3}
    1 < \frac{1}{1-\mathbb{E}_{\gamma,1}\left(-\lambda T^\gamma\right)}.
\end{equation}
Then multiplying both sides of the inequality \eqref{InEq3} by $1-\mathbb{E}_{\gamma,1}\left(-\lambda t^\gamma\right)$ we obtain
\begin{equation*}
    0 < 1-\mathbb{E}_{\gamma,1}\left(-\lambda t^\gamma\right) < \frac{1-\mathbb{E}_{\gamma,1}\left(-\lambda t^\gamma\right)}{1-\mathbb{E}_{\gamma,1}\left(-\lambda T^\gamma\right)} < \frac{1}{1-\mathbb{E}_{\gamma,1}\left(-\lambda T^\gamma\right)} < 1
\end{equation*}
and these inequalities imply \eqref{InEq1}. Rewriting the expression 
\begin{equation*}
    \frac{\mathbb{E}_{\gamma,1}\left(-\lambda T^\gamma\right)-\mathbb{E}_{\gamma,1}\left(-\lambda t^\gamma\right)}{1-\mathbb{E}_{\gamma,1}\left(-\lambda T^\gamma\right)} = \frac{1-\mathbb{E}_{\gamma,1}\left(-\lambda t^\gamma\right)}{1-\mathbb{E}_{\gamma,1}\left(-\lambda T^\gamma\right)} - 1
\end{equation*}
and using the first inequality \eqref{InEq1} we obtain the second inequality \eqref{InEq2}.
\end{proof}

\section{Main Results}
In this Section we deal with the direct problem for the time-fractional pseudo-parabolic equation associated with the Jacobi operator $\Delta_{\alpha,\beta}$ \eqref{EQ: 1.1}. Moreover, ISPs are subject to study. The existence, uniqueness and stability results are established.

\subsection{The direct problem for the time-fractional pseudo-parabolic equation with the Jacobi operator} Let $0<\gamma\leq1$. We consider the non-homogeneous time-fractional pseudo-parabolic equation
\begin{equation}
\label{Heat equation}
\mathbb{D}_{0^+,t}^\gamma\left( u(t,x)-a\Delta_{\alpha,\beta}u(t,x)\right)-\Delta_{\alpha,\beta}u(t,x)+mu(t,x)=f(t,x),\quad (t,x)\in D,
\end{equation}
with initial condition
\begin{equation}\label{Initial date}
u(0,x)=\phi(x),\quad x\in\mathbb{R}^+,
\end{equation}
where the functions $f$ and $\phi$ are given functions. Our aim is to find unique solution $u$ of the problem \eqref{Heat equation} - \eqref{Initial date}.

\begin{theorem}\label{Theorem1} Let $0<\gamma\leq1$ and $\lambda\in\mathbb{R}^+$. Suppose that $f\in C^1([0,T],L^2(\mu))$ and $\phi\in \mathcal{H}$. Then the problem \eqref{Heat equation}-\eqref{Initial date} has
a unique solution $u\in C^\gamma([0,T],L^2(\mu))\cap C([0,T],\mathcal{H})$ and can be represented by formula
\begin{multline*}
    u(t,x)=\int_0^\infty\int_0^\infty\int_0^t (t-\tau)^{\gamma-1} \mathbb{E}_{\gamma,\gamma}\left(-\frac{\lambda^{2}+\rho^2+m}{1+a(\lambda^{2}+\rho^2)}(t-\tau)^\gamma\right)\frac{f(\tau,y)}{1+a(\lambda^{2}+\rho^2)}\\
    \times\varphi_{\lambda}^{\alpha,\beta}(y)\varphi_{\lambda}^{\alpha,\beta}(x)d\tau d\mu_{\alpha,\beta}(y)d\nu_{\alpha,\beta}(\lambda)\\
\end{multline*}
\begin{equation*}
    +\int_0^\infty\int_0^\infty\mathbb{E}_{\gamma,1}\left(-\frac{\lambda^{2}+\rho^2+m}{1+a(\lambda^{2}+\rho^2)}t^\gamma\right)\phi(y)\varphi_{\lambda}^{\alpha,\beta}(y)\varphi_{\lambda}^{\alpha,\beta}(x)d\mu_{\alpha,\beta}(y)d\nu_{\alpha,\beta}(\lambda).
\end{equation*}
\end{theorem}
\begin{proof} We assume that $0<\gamma\leq1$, $\lambda\in\mathbb{R}^+$ and $u(t,\cdot)\in\mathcal{H}$. We first prove that the problem \eqref{Heat equation}-\eqref{Initial date} has only one solution, if the later exists.
Suppose the proposition were false. Assume that there exist two different solutions $u_{1}(t,x)$ and $u_{2}(t,x)$.
Denote $u_{0}(t,x)=u_{1}(t,x)-u_{2}(t,x)$. Then $u_{0}(t,x)$ solves the following equation
\begin{equation}
\label{HomeEquation}
\mathbb{D}_{0^+,t}^\gamma\left( u_0(t,x)-a\Delta_{\alpha,\beta}u_0(t,x)\right)-\Delta_{\alpha,\beta}u_0(t,x)+mu_0(t,x)=0,\quad (t,x)\in D,
\end{equation}
\begin{equation}
\label{ZeroCondition}
u_{0}(0,x)=0,\quad x\in\mathbb{R}^+.
\end{equation}
The problem \eqref{HomeEquation}-\eqref{ZeroCondition} has only trivial solution. This implies uniqueness of the solution.

Now, we will prove the existence of the solutions. Using the Fourier-Jacobi transform
$\mathcal{F}_{\alpha,\beta}$ \eqref{Jacobi Transform} on both sides of \eqref{Heat equation}-\eqref{Initial date}, we have
\begin{equation}
\label{FJEquation}
\mathbb{D}_{0^+,t}^\gamma\widehat{u}(t,\lambda)+\frac{\lambda^{2}+\rho^2+m}{1+a(\lambda^{2}+\rho^2)}\widehat{u}(t,\lambda)=\frac{\widehat{f}(t,\lambda)}{1+a(\lambda^{2}+\rho^2)},
\end{equation}
\begin{equation}
\label{FJCondition}
\widehat{u}(0,\lambda)=\widehat{\phi}(\lambda),
\end{equation}
for all $\lambda\in\mathbb{R}^+$ and $0<t<T$. The solution (see \cite[p. 231, ex. 4.9]{KST06}) of the problem \eqref{FJEquation}-\eqref{FJCondition} is given by 
\begin{multline}\label{JacobiSolution}
    \widehat{u}(t,\lambda) =\int_0^t(t-\tau)^{\gamma-1}\mathbb{E}_{\gamma,\gamma}\left(-\frac{\lambda^{2}+\rho^2+m}{1+a(\lambda^{2}+\rho^2)}(t-\tau)^\gamma\right)\frac{\widehat{f}(\tau,\lambda)}{1+a(\lambda^{2}+\rho^2)}d\tau\\
    + \widehat{\phi}(\lambda)\mathbb{E}_{\gamma,1}\left(-\frac{\lambda^{2}+\rho^2+m}{1+a(\lambda^{2}+\rho^2)}t^\gamma\right),
\end{multline}
where $\mathbb{E}_{\gamma,1}(z)$ is the classical Mittag-Leffler function and $\mathbb{E}_{\gamma,\gamma}(z)$ is the Mittag-Leffler type function. Now by using the inverse Fourier-Jacobi transform $\mathcal{F}_{\alpha,\beta}^{-1}$ \eqref{Invers Jacobi Transform} to \eqref{JacobiSolution}, we obtain the formula for the solution of the problem \eqref{Heat equation}-\eqref{Initial date}, given by
\begin{equation*}
    \begin{split}
       u(t,x)&=\int_0^\infty\int_0^\infty\int_0^t (t-\tau)^{\gamma-1} \mathbb{E}_{\gamma,\gamma}\left(-\frac{\lambda^{2}+\rho^2+m}{1+a(\lambda^{2}+\rho^2)}(t-\tau)^\gamma\right)\frac{f(\tau,y)}{1+a(\lambda^{2}+\rho^2)}\\
       &\times\varphi_{\lambda}^{\alpha,\beta}(y)\varphi_{\lambda}^{\alpha,\beta}(x)d\tau d\mu_{\alpha,\beta}(y)d\nu_{\alpha,\beta}(\lambda)\\
       &+\int_0^\infty\int_0^\infty\mathbb{E}_{\gamma,1}\left(-\frac{\lambda^{2}+\rho^2+m}{1+a(\lambda^{2}+\rho^2)}t^\gamma\right)\phi(y)\varphi_{\lambda}^{\alpha,\beta}(y)\varphi_{\lambda}^{\alpha,\beta}(x)d\mu_{\alpha,\beta}(y)d\nu_{\alpha,\beta}(\lambda).
    \end{split}
\end{equation*}

By using the property
\begin{equation*}
    \frac{d}{d\tau}\left(\mathbb{E}_{\gamma,1}(c \tau^\gamma)\right)=c\tau^{\gamma-1}\mathbb{E}_{\gamma,\gamma}(c \tau^\gamma), \quad c =\text{constant},
\end{equation*}
of the Mittag-Leffler function, we obtain
\begin{multline*}
    \frac{d}{d\tau}\left(\mathbb{E}_{\gamma,1}\left(-\frac{\lambda^{2}+\rho^2+m}{1+a(\lambda^{2}+\rho^2)}(t-\tau)^\gamma\right)\right)\\
    =\frac{\lambda^{2}+\rho^2+m}{1+a(\lambda^{2}+\rho^2)} (t-\tau)^{\gamma-1}\mathbb{E}_{\gamma,\gamma}\left(-\frac{\lambda^{2}+\rho^2+m}{1+a(\lambda^{2}+\rho^2)}(t-\tau)^\gamma\right)
\end{multline*}
and we can write \eqref{JacobiSolution} in a form
\begin{align*}
    \widehat{u}(t,\lambda) &=
    \int_0^t(t-\tau)^{\gamma-1} \mathbb{E}_{\gamma,\gamma}\left(-\frac{\lambda^{2}+\rho^2+m}{1+a(\lambda^{2}+\rho^2)}(t-\tau)^\gamma\right)\frac{\widehat{f}(\tau,\lambda)}{1+a(\lambda^{2}+\rho^2)}d\tau\\
    &+ \widehat{\phi}(\lambda)\mathbb{E}_{\gamma,1}\left(-\frac{\lambda^{2}+\rho^2+m}{1+a(\lambda^{2}+\rho^2)}t^\gamma\right)\\
    &=\frac{1}{\lambda^{2}+\rho^2+m}\int_0^t\frac{d}{d\tau}\left(\mathbb{E}_{\gamma,1}\left(-\frac{\lambda^{2}+\rho^2+m}{1+a(\lambda^{2}+\rho^2)}(t-\tau)^\gamma\right)\right)\widehat{f}(\tau,\lambda)d\tau\\
    &+ \widehat{\phi}(\lambda)\mathbb{E}_{\gamma,1}\left(-\frac{\lambda^{2}+\rho^2+m}{1+a(\lambda^{2}+\rho^2)}t^\gamma\right)\\
    &=\frac{1}{\lambda^{2}+\rho^2+m}\widehat{f}(t,\lambda)-\frac{1}{\lambda^{2}+\rho^2+m}\widehat{f}(0,\lambda)\mathbb{E}_{\gamma,1}\left(-\frac{\lambda^{2}+\rho^2+m}{1+a(\lambda^{2}+\rho^2)}t^\gamma\right)\\
    &-\frac{1}{\lambda^{2}+\rho^2+m}\int_0^t\mathbb{E}_{\gamma,1}\left(-\frac{\lambda^{2}+\rho^2+m}{1+a(\lambda^{2}+\rho^2)}(t-\tau)^\gamma\right)\frac{d}{d\tau}\widehat{f}(\tau,\lambda)d\tau\\
    &+ \widehat{\phi}(\lambda)\mathbb{E}_{\gamma,1}\left(-\frac{\lambda^{2}+\rho^2+m}{1+a(\lambda^{2}+\rho^2)}t^\gamma\right)
\end{align*}
by using the rule integration by parts and $\mathbb{E}_{\gamma,1}(0)=1$. Let $0<\gamma<1$ and $f\in C^1([0,T],L^2(\mu))$, $\phi\in\mathcal{H}$, then we can estimate $u$ as follows
\begin{align*}
    \|u(t,\cdot)\|_{\mathcal{H}}^2&=\int_0^\infty\left|(\lambda^{2}+\rho^2)\widehat{u}(t,\lambda)\right|^2d\nu_{\alpha,\beta}(\lambda)
\end{align*}
\begin{align*}
    &\lesssim\int_0^\infty\left|(\lambda^{2}+\rho^2)\frac{\widehat{f}(t,\lambda)}{\lambda^{2}+\rho^2+m}\right|^2d\nu_{\alpha,\beta}(\lambda)\\
    &+\int_0^\infty\left|(\lambda^{2}+\rho^2)\frac{\widehat{f}(0,\lambda)}{\lambda^{2}+\rho^2+m}\mathbb{E}_{\gamma,1}\left(-\frac{\lambda^{2}+\rho^2+m}{1+a(\lambda^{2}+\rho^2)}t^\gamma\right)\right|^2d\nu_{\alpha,\beta}(\lambda)\\
    &+\int_0^\infty\left|\frac{\lambda^{2}+\rho^2}{\lambda^{2}+\rho^2+m}\int_0^t\mathbb{E}_{\gamma,1}\left(-\frac{\lambda^{2}+\rho^2+m}{1+a(\lambda^{2}+\rho^2)}(t-\tau)^\gamma\right)\frac{d}{d\tau}\widehat{f}(\tau,\lambda)d\tau\right|^2\\
    &\times d\nu_{\alpha,\beta}(\lambda) + \int_0^\infty\left|(\lambda^{2}+\rho^2)\widehat{\phi}(\lambda)\mathbb{E}_{\gamma,1}\left(-\frac{\lambda^{2}+\rho^2+m}{1+a(\lambda^{2}+\rho^2)}t^\gamma\right)\right|^2d\nu_{\alpha,\beta}(\lambda)\\
    &\lesssim\int_0^\infty\left|\widehat{f}(t,\lambda)\right|^2d\nu_{\alpha,\beta}(\lambda)+\int_0^\infty\left|\widehat{f}(0,\lambda)\right|^2d\nu_{\alpha,\beta}(\lambda)\\
    &+\int_0^\infty\left(\int_0^t\left|\frac{d}{d\tau}\widehat{f}(\tau,\lambda)\right|d\tau\right)^2d\nu_{\alpha,\beta}(\lambda)+ \int_0^\infty\left|(\lambda^{2}+\rho^2)\widehat{\phi}(\lambda)\right|^2d\nu_{\alpha,\beta}(\lambda)\\
    &\lesssim \|f(t,\cdot)\|_{2,\mu}^2 + \|f(0,\cdot)\|_{2,\mu}^2 + \int_0^T \|\frac{d}{dt}f(t,\cdot)\|_{2,\mu}^2 dt + \|\phi\|_{\mathcal{H}}^2,
\end{align*}
here we have used Cauchy-Schwarz inequality, Fubibi's theorem and $a\lesssim b$ denotes $a\leq cb$ for some positive constant $c$ independent of $a$ and $b$. Thus,
\begin{equation*}
    \|u(t,\cdot)\|_{\mathcal{H}}^2 \lesssim \|f(t,\cdot)\|_{2,\mu}^2 + \|f(0,\cdot)\|_{2,\mu}^2 + \int_0^T \|\frac{d}{dt}f(t,\cdot)\|_{2,\mu}^2 dt + \|\phi\|_{\mathcal{H}}^2.
\end{equation*}
Then, we obtain
\begin{equation*}
    \|u\|_{C([0,T],\mathcal{H})}^2\lesssim \|f\|_{C^1([0,T],L^2(\mu))}^2 + \|\phi\|_{\mathcal{H}}^2<\infty.
\end{equation*}

In a case $\gamma=1$, we have
\begin{equation*}
    \|u(t,\cdot)\|_{\mathcal{H}}^2=\int_0^\infty\left|(\lambda^{2}+\rho^2)\widehat{u}(t,\lambda)\right|^2d\nu_{\alpha,\beta}(\lambda)
\end{equation*}
\begin{multline*}
    =\int_0^\infty\Biggl|(\lambda^{2}+\rho^2)\int_0^t \frac{\widehat{f}(\tau,\lambda)}{1+a(\lambda^2+\rho^2)}e^{-\frac{\lambda^{2}+\rho^2+m}{1+a(\lambda^{2}+\rho^2)}(t-\tau)}d\tau\\
    +(\lambda^{2}+\rho^2)\widehat{\phi}(\lambda)e^{-\frac{\lambda^{2}+\rho^2+m}{1+a(\lambda^{2}+\rho^2)}t}\Biggr|^2d\nu_{\alpha,\beta}(\lambda)
\end{multline*}
\begin{multline*}
    \lesssim\int_0^\infty\left|\int_0^t \widehat{f}(\tau,\lambda)e^{-\frac{\lambda^{2}+\rho^2+m}{1+a(\lambda^{2}+\rho^2)}(t-\tau)}d\tau\right|^2d\nu_{\alpha,\beta}(\lambda)\\
    +\int_0^\infty\left|(\lambda^{2}+\rho^2)\widehat{\phi}(\lambda)e^{-\frac{\lambda^{2}+\rho^2+m}{1+a(\lambda^{2}+\rho^2)}t}\right|^2d\nu_{\alpha,\beta}(\lambda)
\end{multline*}
\begin{multline*}
    \lesssim\int_0^\infty\int_0^T\left|\widehat{f}(t,\lambda)\right|^2dt d\nu_{\alpha,\beta}(\lambda) + \int_0^\infty\left|(\lambda^{2}+\rho^2)\widehat{\phi}(\lambda)\right|^2d\nu_{\alpha,\beta}(\lambda)\\
    =\int_0^T \|f(t,\cdot)\|_{2,\mu}^2dt + \|\phi\|_{\mathcal{H}}^2
\end{multline*}
by using Cauchy-Schwarz inequality and Fubini's theorem. Thus,
\begin{equation*}
    \|u(t,\cdot)\|_{\mathcal{H}}^2\lesssim\int_0^T \|f(t,\cdot)\|_{2,\mu}^2dt + \|\phi\|_{\mathcal{H}}^2.
\end{equation*}
Then, we have 
\begin{equation*}
    \|u\|_{C([0,T],\mathcal{H})}^2 \lesssim\|f\|_{C([0,T],L^2(\mu))}^2 + \|\phi\|_{\mathcal{H}}^2<\infty.
\end{equation*}
Let us estimate the function $\mathbb{D}_{0^+,t}^\gamma u$
\begin{multline*}
    \|\mathbb{D}_{0^+,t}^\gamma u(t,\cdot)\|_{2,\mu}^2 = \|\mathbb{D}_{0^+,t}^\gamma \widehat{u}(t,\cdot)\|_{2,\nu}^2 = \int_0^\infty\left|\mathbb{D}_{0^+,t}^\gamma \widehat{u}(t,\cdot)\right|^2d\nu_{\alpha,\beta}(\lambda)\\ =\int_0^\infty\left|\frac{\widehat{f}(t,\lambda)}{1+a(\lambda^{2}+\rho^2)}-\frac{\lambda^{2}+\rho^2+m}{1+a(\lambda^{2}+\rho^2)}\widehat{u}(t,\lambda)\right|^2d\nu_{\alpha,\beta}(\lambda)\\
    \lesssim \|\widehat{f}(t,\cdot)\|_{2,\nu}^2 + \|\widehat{u}(t,\cdot)\|_{2,\nu}^2.
\end{multline*}
Thus, we have
\begin{equation*}
    \|\mathbb{D}_{0^+,t}^\gamma u(t,\cdot)\|_{2,\mu}^2 \lesssim \|f(t,\cdot)\|_{2,\mu}^2 + \|u(t,\cdot)\|_{2,\mu}^2
\end{equation*}
and
\begin{equation*}
    \|\mathbb{D}_{0^+,t}^\gamma u\|_{C([0,T],L^2(\mu))}^2 \lesssim \|f\|_{C([0,T],L^2(\mu))}^2 +\|u\|_{C([0,T],L^2(\mu))}^2<\infty.
\end{equation*}
Consequently, using Definition \ref{DefinitionSpace} we obtain $u\in C^\gamma([0,T],L^2(\mu))$. Our prove is completed.
\end{proof}

\subsection{The ISP for the time-fractional pseudo-parabolic equation} This subsection deals with the ISP for the time-fractional pseudo-parabolic equation associated with the Jacobi operator $\Delta_{\alpha,\beta}$ \eqref{EQ: 1.1}.

\subsubsection{Statement of the problem} Let $0<\gamma\leq1$. We  aim  to  find  a  couple  of  functions  $(u,f)$  satisfying equation 
\begin{equation}\label{InvHeat equation}
\mathbb{D}_{0^+,t}^\gamma\left( u(t,x)-a\Delta_{\alpha,\beta}u(t,x)\right)-\Delta_{\alpha,\beta}u(t,x)+mu(t,x)=f(x),\quad (t,x)\in D,
\end{equation}
under conditions
\begin{equation}\label{Date1}
u(0,x)=\phi(x),\quad x\in\mathbb{R}^+,
\end{equation}
and
\begin{equation}\label{Date2}
u(T,x)=\psi(x),\quad x\in\mathbb{R}^+.
\end{equation}

\begin{theorem}\label{Theorem2} Let $0<\gamma\leq1$. Assume that $\psi,\phi\in\mathcal{H}$. Then the pair $(u,f)$ is a unique solution of the ISP \eqref{InvHeat equation}-\eqref{Date2}, which are functions $u\in C^\gamma([0,T],L^{2}(\mu))\cap C([0,T],\mathcal{H}), f\in L^{2}(\mu)$ can be represented by the formulas
\begin{multline*}
    u(t,x)=\int_0^\infty\int_0^\infty\frac{1-\mathbb{E}_{\gamma,1}\left(-\frac{\lambda^{2}+\rho^2+m}{1+a(\lambda^{2}+\rho^2)}t^\gamma\right)}{1-\mathbb{E}_{\gamma,1}\left(-\frac{\lambda^{2}+\rho^2+m}{1+a(\lambda^{2}+\rho^2)}T^\gamma\right)}\psi(y)\varphi_{\lambda}^{\alpha,\beta}(y)\varphi_{\lambda}^{\alpha,\beta}(x)d\mu_{\alpha,\beta}(y)d\nu_{\alpha,\beta}(\lambda)\\
    -\int_0^\infty\int_0^\infty\frac{\mathbb{E}_{\gamma,1}\left(-\frac{\lambda^{2}+\rho^2+m}{1+a(\lambda^{2}+\rho^2)}T^\gamma\right)-\mathbb{E}_{\gamma,1}\left(-\frac{\lambda^{2}+\rho^2+m}{1+a(\lambda^{2}+\rho^2)}t^\gamma\right)}{1-\mathbb{E}_{\gamma,1}\left(-\frac{\lambda^{2}+\rho^2+m}{1+a(\lambda^{2}+\rho^2)}T^\gamma\right)}\\
    \times\phi(y)\varphi_{\lambda}^{\alpha,\beta}(y)\varphi_{\lambda}^{\alpha,\beta}(x)d\mu_{\alpha,\beta}(y)d\nu_{\alpha,\beta}(\lambda)
\end{multline*}
and
\begin{multline*}
    f(x)=\int_0^\infty\int_0^\infty(\lambda^{2}+\rho^2+m)\frac{\psi(y)-\phi(y)\mathbb{E}_{\gamma,1}\left(-\frac{\lambda^{2}+\rho^2+m}{1+a(\lambda^{2}+\rho^2)}T^\gamma\right)}{1-\mathbb{E}_{\gamma,1}\left(-\frac{\lambda^{2}+\rho^2+m}{1+a(\lambda^{2}+\rho^2)}T^\gamma\right)}\\
    \times\varphi_{\lambda}^{\alpha,\beta}(y)\varphi_{\lambda}^{\alpha,\beta}(x)d\mu_{\alpha,\beta}(y)d\nu_{\alpha,\beta}(\lambda).
\end{multline*}
\end{theorem}
\begin{proof}
We assume that $0<\gamma\leq1$, and $u(t,\cdot),f\in\mathcal{H}$. Let us first prove the existence result. By using the Fourier-Jacobi transform
$\mathcal{F}_{\alpha,\beta}$ \eqref{Jacobi Transform} on both sides of \eqref{InvHeat equation}-\eqref{Date2}, we obtain
\begin{equation}
\label{FJInvHeat equation}
\mathbb{D}_{0^+,t}^\gamma\widehat{u}(t,\lambda)+\frac{\lambda^{2}+\rho^2+m}{1+a(\lambda^{2}+\rho^2)}\widehat{u}(t,\lambda)=\frac{\widehat{f}(\lambda)}{1+a(\lambda^{2}+\rho^2)},\quad (t,\lambda)\in D,
\end{equation}
\begin{equation}
\label{FJDate1}
\widehat{u}(0,\lambda)=\widehat{\phi}(\lambda),\quad \lambda\in\mathbb{R}^+,
\end{equation}
\begin{equation}
\label{FJDate2}
\widehat{u}(T,\lambda)=\widehat{\psi}(\lambda),\quad \lambda\in\mathbb{R}^+.
\end{equation}
Solution of the equation \eqref{FJInvHeat equation} is given by 
\begin{equation} \label{Solution11}
    \widehat{u}(t,\lambda)=\frac{\widehat{f}(\lambda)}{\lambda^{2}+\rho^2+m} + C(\lambda)\mathbb{E}_{\gamma,1}\left(-\frac{\lambda^{2}+\rho^2+m}{1+a(\lambda^{2}+\rho^2)}t^\gamma\right),
\end{equation}
for all $0<\gamma\leq1$ and functions $\widehat{f}(\lambda)$ and $C(\lambda)$ are unknown functions. For determine these functions we use conditions \eqref{FJDate1} and \eqref{FJDate2}. After that we have
\begin{equation*}
    \widehat{u}(0,\lambda)=\frac{\widehat{f}(\lambda)}{\lambda^{2}+\rho^2+m} + C(\lambda)=\widehat{\phi}(\lambda)
\end{equation*}
and
\begin{equation*}
    \widehat{u}(T,\lambda)=\frac{\widehat{f}(\lambda)}{\lambda^{2}+\rho^2+m} + C(\lambda)\mathbb{E}_{\gamma,1}\left(-\frac{\lambda^{2}+\rho^2+m}{1+a(\lambda^{2}+\rho^2)}T^\gamma\right)=\widehat{\psi}(\lambda).
\end{equation*}
Thus we have
\begin{equation*}
    C(\lambda) = \frac{\widehat{\phi}(\lambda)-\widehat{\psi}(\lambda)}{1-\mathbb{E}_{\gamma,1}\left(-\frac{\lambda^{2}+\rho^2+m}{1+a(\lambda^{2}+\rho^2)}T^\gamma\right)}
\end{equation*}
and
\begin{equation} \label{f solution}
    \widehat{f}(\lambda)=(\lambda^{2}+\rho^2+m)\frac{\widehat{\psi}(\lambda)-\widehat{\phi}(\lambda)\mathbb{E}_{\gamma,1}\left(-\frac{\lambda^{2}+\rho^2+m}{1+a(\lambda^{2}+\rho^2)}T^\gamma\right)}{1-\mathbb{E}_{\gamma,1}\left(-\frac{\lambda^{2}+\rho^2+m}{1+a(\lambda^{2}+\rho^2)}T^\gamma\right)}.
\end{equation}
Substituting the resulting functions $C(\lambda)$ and $\widehat{f}(\lambda)$ into \eqref{Solution11}, we get
\begin{multline*}
    \widehat{u}(t,\lambda)=\frac{1-\mathbb{E}_{\gamma,1}\left(-\frac{\lambda^{2}+\rho^2+m}{1+a(\lambda^{2}+\rho^2)}t^\gamma\right)}{1-\mathbb{E}_{\gamma,1}\left(-\frac{\lambda^{2}+\rho^2+m}{1+a(\lambda^{2}+\rho^2)}T^\gamma\right)}\widehat{\psi}(\lambda)\\
    -\frac{\mathbb{E}_{\gamma,1}\left(-\frac{\lambda^{2}+\rho^2+m}{1+a(\lambda^{2}+\rho^2)}T^\gamma\right)-\mathbb{E}_{\gamma,1}\left(-\frac{\lambda^{2}+\rho^2+m}{1+a(\lambda^{2}+\rho^2)}t^\gamma\right)}{1-\mathbb{E}_{\gamma,1}\left(-\frac{\lambda^{2}+\rho^2+m}{1+a(\lambda^{2}+\rho^2)}T^\gamma\right)}\widehat{\phi}(\lambda).
\end{multline*}
Therefore solution of the problem \eqref{FJInvHeat equation} - \eqref{FJDate2} is the pair $(\widehat{u},\widehat{f})$. We obtain solution of the problem \eqref{InvHeat equation}-\eqref{Date2} by applying the inverse Fourier-Jacobi transform $\mathcal{F}_{\alpha,\beta}^{-1}$ \eqref{Invers Jacobi Transform} to the functions $\widehat{u}$ and $\widehat{f}$, i.e.  
\begin{multline}\label{Sol1}
    u(t,x)=\int_0^\infty\int_0^\infty\frac{1-\mathbb{E}_{\gamma,1}\left(-\frac{\lambda^{2}+\rho^2+m}{1+a(\lambda^{2}+\rho^2)}t^\gamma\right)}{1-\mathbb{E}_{\gamma,1}\left(-\frac{\lambda^{2}+\rho^2+m}{1+a(\lambda^{2}+\rho^2)}T^\gamma\right)}\psi(y)\varphi_{\lambda}^{\alpha,\beta}(y)\varphi_{\lambda}^{\alpha,\beta}(x)d\mu_{\alpha,\beta}(y)d\nu_{\alpha,\beta}(\lambda)\\
    -\int_0^\infty\int_0^\infty\frac{\mathbb{E}_{\gamma,1}\left(-\frac{\lambda^{2}+\rho^2+m}{1+a(\lambda^{2}+\rho^2)}T^\gamma\right)-\mathbb{E}_{\gamma,1}\left(-\frac{\lambda^{2}+\rho^2+m}{1+a(\lambda^{2}+\rho^2)}t^\gamma\right)}{1-\mathbb{E}_{\gamma,1}\left(-\frac{\lambda^{2}+\rho^2+m}{1+a(\lambda^{2}+\rho^2)}T^\gamma\right)}\\
    \times\phi(y)\varphi_{\lambda}^{\alpha,\beta}(y)\varphi_{\lambda}^{\alpha,\beta}(x)d\mu_{\alpha,\beta}(y)d\nu_{\alpha,\beta}(\lambda)
\end{multline}
and
\begin{multline}\label{Sol2}
    f(x)=\int_0^\infty\int_0^\infty(\lambda^{2}+\rho^2+m)\frac{\psi(y)-\phi(y)\mathbb{E}_{\gamma,1}\left(-\frac{\lambda^{2}+\rho^2+m}{1+a(\lambda^{2}+\rho^2)}T^\gamma\right)}{1-\mathbb{E}_{\gamma,1}\left(-\frac{\lambda^{2}+\rho^2+m}{1+a(\lambda^{2}+\rho^2)}T^\gamma\right)}\\
    \times\varphi_{\lambda}^{\alpha,\beta}(y)\varphi_{\lambda}^{\alpha,\beta}(x)d\mu_{\alpha,\beta}(y)d\nu_{\alpha,\beta}(\lambda),
\end{multline}
for all $0<\gamma\leq1$. 

Let $\psi,\phi\in\mathcal{H}$. Then using Lemma \ref{lemma} we can estimate the function $u$ as following 
\begin{equation*}
     \|u(t,\cdot)\|_{\mathcal{H}}^2 = \int_0^\infty|(\lambda^2+\rho^2)\widehat{u}(t,\lambda)|^2d\nu_{\alpha,\beta}(\lambda)
\end{equation*}
\begin{align*}
    &\lesssim \int_0^\infty\left|(\lambda^2+\rho^2)\widehat{\psi}(\lambda)\frac{1-\mathbb{E}_{\gamma,1}\left(-\frac{\lambda^{2}+\rho^2+m}{1+a(\lambda^{2}+\rho^2)}t^\gamma\right)}{1-\mathbb{E}_{\gamma,1}\left(-\frac{\lambda^{2}+\rho^2+m}{1+a(\lambda^{2}+\rho^2)}T^\gamma\right)}\right|^2d\nu_{\alpha,\beta}(\lambda)\\
    &+ \int_0^\infty\left|(\lambda^2+\rho^2)\widehat{\phi}(\lambda)\frac{\mathbb{E}_{\gamma,1}\left(-\frac{\lambda^{2}+\rho^2+m}{1+a(\lambda^{2}+\rho^2)}T^\gamma\right)-\mathbb{E}_{\gamma,1}\left(-\frac{\lambda^{2}+\rho^2+m}{1+a(\lambda^{2}+\rho^2)}t^\gamma\right)}{1-\mathbb{E}_{\gamma,1}\left(-\frac{\lambda^{2}+\rho^2+m}{1+a(\lambda^{2}+\rho^2)}T^\gamma\right)}\right|^2d\nu_{\alpha,\beta}(\lambda)\\ 
    &\lesssim \int_0^\infty|(\lambda^2+\rho^2)\widehat{\psi}(\lambda)|^2d\nu_{\alpha,\beta}(\lambda) + \int_0^\infty|(\lambda^2+\rho^2)\widehat{\phi}(\lambda)|^2d\nu_{\alpha,\beta}(\lambda).
\end{align*}
Thus,
\begin{equation*}
    \|u(t,\cdot)\|_{\mathcal{H}}^2\lesssim\|\psi\|_{\mathcal{H}}^2 + \|\phi\|_{\mathcal{H}}^2<\infty.
\end{equation*}
Then we have
\begin{equation*}
       \|u\|_{C([0,T],\mathcal{H})} \lesssim \|\psi\|_{\mathcal{H}} + \|\phi\|_{\mathcal{H}}<\infty.
\end{equation*}
Let us estimate the function $f$
\begin{align*}
\|f\|_{2,\mu}^2&=\|\widehat{f}\|_{2,\nu}^2=\int_0^\infty |\widehat{f}(\lambda)|^2d\nu_{\alpha,\beta}(\lambda)\\
&=\int_0^\infty\left|(\lambda^{2}+\rho^2+m)\frac{\widehat{\psi}(\lambda)-\widehat{\phi}(\lambda)\mathbb{E}_{\gamma,1}\left(-\frac{\lambda^{2}+\rho^2+m}{1+a(\lambda^{2}+\rho^2)}T^\gamma\right)}{1-\mathbb{E}_{\gamma,1}\left(-\frac{\lambda^{2}+\rho^2+m}{1+a(\lambda^{2}+\rho^2)}T^\gamma\right)}\right|^2d\nu_{\alpha,\beta}(\lambda)\\
&\lesssim\int_0^\infty\left|(\lambda^2+\rho^2+m)\frac{\widehat{\psi}(\lambda)}{1-\mathbb{E}_{\gamma,1}\left(-\frac{\lambda^{2}+\rho^2+m}{1+a(\lambda^{2}+\rho^2)}T^\gamma\right)}\right|^2d\nu_{\alpha,\beta}(\lambda)\\
&+\int_0^\infty\left|(\lambda^2+\rho^2+m)\frac{\widehat{\phi}(\lambda)\mathbb{E}_{\gamma,1}\left(-\frac{\lambda^{2}+\rho^2+m}{1+a(\lambda^{2}+\rho^2)}T^\gamma\right)}{1-\mathbb{E}_{\gamma,1}\left(-\frac{\lambda^{2}+\rho^2+m}{1+a(\lambda^{2}+\rho^2)}T^\gamma\right)}\right|^2d\nu_{\alpha,\beta}(\lambda)\\
&\lesssim\|\psi\|_{\mathcal{H}}^2 + \|\phi\|_{\mathcal{H}}^2.
\end{align*}
So, we obtain
\begin{equation*}
    \|f\|_{2,\mu}^2 \lesssim \|\psi\|_{\mathcal{H}}^2 + \|\phi\|_{\mathcal{H}}^2<\infty.
\end{equation*}
Next, we estimate the function $\mathbb{D}_{0^+,t}^\gamma u$
\begin{multline*}
\|\mathbb{D}_{0^+,t}^\gamma u(t,\cdot)\|_{2,\mu}^2 = \|\mathbb{D}_{0^+,t}^\gamma\widehat{u}(t,\cdot)\|_{2,\nu}^2 = \int_0^\infty|\mathbb{D}_{0^+,t}^\gamma\widehat{u}(t,\lambda)|^2d\nu_{\alpha,\beta}(\lambda)\\
= \int_0^\infty\left|\frac{\widehat{f}(\lambda)}{1+a(\lambda^{2}+\rho^2)} - \frac{\lambda^{2}+\rho^2+m}{1+a(\lambda^{2}+\rho^2)}\widehat{u}(t,\lambda)\right|^2d\nu_{\alpha,\beta}(\lambda)\\
\lesssim \|f\|_{2,\nu}^2 + \|u(t,\cdot)\|_{2,\nu}^2.
\end{multline*}
Finally, we have
\begin{equation*}
    \|\mathbb{D}_{0^+,t}^\gamma u\|_{C([0,T],L^2(\mu))}^2 \lesssim \|f\|_{2,\mu}^2 + \|u\|_{C([0,T],L^2(\mu))}^2
    < \infty.
\end{equation*}
It is obvious that $\|u\|_{C([0,T],L^2(\mu))}^2< \infty$. The existence is proved.

Now, let us prove the uniqueness of the solution. Taking into account the property of the Fourier-Jacobi transform Proposition \ref{Norm-preserving map}, one observes that a pair of functions $(u,f)$ is uniquely determined by the formulas \eqref{Sol1} and \eqref{Sol2}. The uniqueness is proved.
\end{proof}

\subsubsection{Stability Theorem}
Finally, we study a stability property of the solution $(u,f)$ of the problem \eqref{InvHeat equation}-\eqref{Date2} given by the formulas \eqref{Sol1} and \eqref{Sol2}, .

\begin{theorem} 
\label{Stability} 
Let $(u,f)$ and $(u_d,f_d)$ be solutions of the problem \eqref{InvHeat equation}-\eqref{Date2} corresponding to the
data $(\phi,\psi)$ and its small perturbation $(\phi_d,\psi_d)$, respectively. Then the solution of the problem \eqref{InvHeat equation}-\eqref{Date2} depends continuously on these data, namely, we have
\begin{equation*}
    \|u-u_d\|_{C([0,T],\mathcal{H})}^2 \lesssim \|\psi-\psi_d\|_{\mathcal{H}}^2 + \|\phi-\phi_d\|_{\mathcal{H}}^2
\end{equation*}
and
\begin{equation*}
    \|f-f_d\|_{2,\mu}^2 \lesssim \|\psi-\psi_d\|_{\mathcal{H}}^2 + \|\phi-\phi_d\|_{\mathcal{H}}^2.
\end{equation*}
\end{theorem}

\begin{proof}
From the definition of the Fourier-Jacobi transform \eqref{Jacobi Transform}
\begin{equation*}
\mathcal{F}_{\alpha,\beta}(u(t,\cdot))(\lambda)=\widehat{u}(t,\lambda)=\int_0^\infty u(t,x)\varphi_{\lambda}^{\alpha,\beta}(x) d\mu_{\alpha,\beta}(x),
\end{equation*}
we conclude that
\begin{align*}
    \mathcal{F}_{\alpha,\beta}(u(t,\cdot)-u_d(t,\cdot))(\lambda) &= \int_0^\infty (u(t,x)-u_d(t,x))\varphi_{\lambda}^{\alpha,\beta}(x) d\mu_{\alpha,\beta}(x)\\
    &= \int_0^\infty u(t,x)\varphi_{\lambda}^{\alpha,\beta}(x) d\mu_{\alpha,\beta}(x)\\
    &- \int_0^\infty u_d(t,x)\varphi_{\lambda}^{\alpha,\beta}(x) d\mu_{\alpha,\beta}(x)\\
    &= \mathcal{F}_{\alpha,\beta}(u(t,\cdot))(\lambda) - \mathcal{F}_{\alpha,\beta}(u_d(t,\cdot))(\lambda)\\
    &=\widehat{u}(t,\lambda)-\widehat{u}_d(t,\lambda),
\end{align*}
here we have used property of the integral. According to the above statement and using Lemma \ref{lemma}, we have
\begin{equation*}
    \|u(t,\cdot)-u_d(t,\cdot)\|_{\mathcal{H}}^2 = \int_0^\infty (\lambda^2+\rho^2)^2|\widehat{u}(t,\lambda)-\widehat{u}_d(t,\lambda)|^2 d\nu_{\alpha,\beta}(\lambda)
\end{equation*}
\begin{align*}
    &=\int_0^\infty (\lambda^2+\rho^2)^2\Bigg|\frac{1-\mathbb{E}_{\gamma,1}\left(-\frac{\lambda^{2}+\rho^2+m}{1+a(\lambda^{2}+\rho^2)}t^\gamma\right)}{1-\mathbb{E}_{\gamma,1}\left(-\frac{\lambda^{2}+\rho^2+m}{1+a(\lambda^{2}+\rho^2)}T^\gamma\right)}\left(\widehat{\psi}(\lambda)-\widehat{\psi}_d(\lambda)\right)\\
    &-\frac{\mathbb{E}_{\gamma,1}\left(-\frac{\lambda^{2}+\rho^2+m}{1+a(\lambda^{2}+\rho^2)}T^\gamma\right)-\mathbb{E}_{\gamma,1}\left(-\frac{\lambda^{2}+\rho^2+m}{1+a(\lambda^{2}+\rho^2)}t^\gamma\right)}{1-\mathbb{E}_{\gamma,1}\left(-\frac{\lambda^{2}+\rho^2+m}{1+a(\lambda^{2}+\rho^2)}T^\gamma\right)}\left(\widehat{\phi}(\lambda)-\widehat{\phi}_d(\lambda)\right)\bigg|^2d\nu_{\alpha,\beta}(\lambda)\\
    &\lesssim \int_0^\infty (\lambda^2+\rho^2)^2\left|\widehat{\psi}(\lambda)-\widehat{\psi}_d(\lambda)\right|^2d\nu_{\alpha,\beta}(\lambda) + \int_0^\infty (\lambda^2+\rho^2)^2\left|\widehat{\phi}(\lambda)-\widehat{\phi}_d(\lambda)\right|^2d\nu_{\alpha,\beta}(\lambda)\\
    &= \|\psi-\psi_d\|_{\mathcal{H}}^2 + \|\phi-\phi_d\|_{\mathcal{H}}^2.
\end{align*}
Thus, one gets
\begin{equation*}
    \|u(t,\cdot)-u_d(t,\cdot)\|_{\mathcal{H}}^2 \lesssim \|\psi-\psi_d\|_{\mathcal{H}}^2 + \|\phi-\phi_d\|_{\mathcal{H}}^2,
\end{equation*}
and
\begin{equation*}
    \|u-u_d\|_{C([0,T],\mathcal{H})}^2 \lesssim \|\psi-\psi_d\|_{\mathcal{H}}^2 + \|\phi-\phi_d\|_{\mathcal{H}}^2.
\end{equation*}

By writing \eqref{f solution} in the form
\begin{equation*}
    \widehat{f}(\lambda) = \frac{\lambda^{2}+\rho^2+m}{1 - \mathbb{E}_{\gamma,1}\left(-\frac{\lambda^{2}+\rho^2+m}{1+a(\lambda^{2}+\rho^2)}T^\gamma\right)}\widehat{\psi}(\lambda) -\frac{(\lambda^{2}+\rho^2+m)\mathbb{E}_{\gamma,1}\left(-\frac{\lambda^{2}+\rho^2+m}{1+a(\lambda^{2}+\rho^2)}T^\gamma\right)}{1 - \mathbb{E}_{\gamma,1}\left(-\frac{\lambda^{2}+\rho^2+m}{1+a(\lambda^{2}+\rho^2)}T^\gamma\right)}\widehat{\phi}(\lambda),
\end{equation*}
and applying similar estimates again we can observe that 
\begin{multline*}
    \|f-f_d\|_{2,\mu}^2 = \|\widehat{f}-\widehat{f_d}\|_{2,\nu}^2 = \int_0^\infty \left|\widehat{f}(\lambda)-\widehat{f_d}(\lambda)\right|^2 d\nu_{\alpha,\beta}(\lambda)\\
    =\int_0^\infty \Biggl|\frac{\lambda^{2}+\rho^2+m}{1 - \mathbb{E}_{\gamma,1}\left(-\frac{\lambda^{2}+\rho^2+m}{1+a(\lambda^{2}+\rho^2)}T^\gamma\right)}\left(\widehat{\psi}(\lambda)-\widehat{\psi}_d(\lambda)\right)\\
    - \frac{(\lambda^{2}+\rho^2+m)\mathbb{E}_{\gamma,1}\left(-\frac{\lambda^{2}+\rho^2+m}{1+a(\lambda^{2}+\rho^2)}T^\gamma\right)}{1 - \mathbb{E}_{\gamma,1}\left(-\frac{\lambda^{2}+\rho^2+m}{1+a(\lambda^{2}+\rho^2)}T^\gamma\right)}\left(\widehat{\phi}(\lambda)-\widehat{\phi}_d(\lambda)\right)\Biggr|^2d\nu_{\alpha,\beta}(\lambda)
\end{multline*}
\begin{align*}
    &\lesssim \int_0^\infty \left|\frac{\lambda^{2}+\rho^2+m}{1 - \mathbb{E}_{\gamma,1}\left(-\frac{\lambda^{2}+\rho^2+m}{1+a(\lambda^{2}+\rho^2)}T^\gamma\right)}\left(\widehat{\psi}(\lambda)-\widehat{\psi}_d(\lambda)\right)\right|^2d\nu_{\alpha,\beta}(\lambda)\\
    &+ \int_0^\infty \left|\frac{(\lambda^{2}+\rho^2+m)\mathbb{E}_{\gamma,1}\left(-\frac{\lambda^{2}+\rho^2+m}{1+a(\lambda^{2}+\rho^2)}T^\gamma\right)}{1 - \mathbb{E}_{\gamma,1}\left(-\frac{\lambda^{2}+\rho^2+m}{1+a(\lambda^{2}+\rho^2)}T^\gamma\right)}\left(\widehat{\phi}(\lambda)-\widehat{\phi}_d(\lambda)\right)\right|^2d\nu_{\alpha,\beta}(\lambda)\\
    &\lesssim \int_0^\infty (\lambda^{2}+\rho^2)^2 \left|\widehat{\psi}(\lambda)-\widehat{\psi}_d(\lambda)\right|^2d\nu_{\alpha,\beta}(\lambda) + \int_0^\infty (\lambda^{2}+\rho^2)^2 \left|\widehat{\phi}(\lambda)-\widehat{\phi}_d(\lambda)\right|^2d\nu_{\alpha,\beta}(\lambda)\\
    &= \|\psi-\psi_d\|_{\mathcal{H}}^2 + \|\phi-\phi_d\|_{\mathcal{H}}^2.
\end{align*}
It follows easily that
\begin{equation*}
    \|f-f_d\|_{2,\mu}^2 \lesssim \|\psi-\psi_d\|_{\mathcal{H}}^2 + \|\phi-\phi_d\|_{\mathcal{H}}^2,
\end{equation*}
ending the proof.
\end{proof}

\subsubsection{Stability Test} 
Here to check Theorem \ref{Stability} we consider a ISP for the heat equation with one dimensional Sturm-Liouville operator
\begin{equation}
\label{Eqq1} 
u_t(t,x)-u_{xx}(t,x) = f(x), \quad 0<t<1, \quad x>0,
\end{equation} 
with conditions 
\begin{equation}
\label{Conn1} 
u(0,x)=u(1,x)=0, 
\end{equation} 
where we put $T=\gamma=1$, $\alpha=\beta=-\frac{1}{2}$, $a=m=0$ and $\phi(x)=\psi(x)=0$ for all $x>0$.

%In a case $\alpha=\beta=-\frac{1}{2}$, the Jacobi operator is $\Delta_{-\frac{1}{2}, -\frac{1}{2}}=\frac{d^2}{dt^2}$ and the Jacobi function is $\varphi_{\lambda}^{-\frac{1}{2}, -\frac{1}{2}}(x)=\cos(\lambda x)$, so we obtain for all $f\in L^1(\mathbb{R})$ the Fourier-cosine transform\begin{equation*}\widehat{f}(\lambda)= \frac{1}{\sqrt{2\pi}} \int_0^\infty f(x)\cos(\lambda x)dx \end{equation*}with $d\mu_{-\frac{1}{2}, -\frac{1}{2}}(x)=\frac{1}{\sqrt{2\pi}}dx$ and for all $\widehat{f}\in L^1(\mathbb{R})$ the inverse Fourier-cosine transform\begin{equation*}f(x)= \frac{4}{\sqrt{2\pi}} \int_0^\infty \widehat{f}(\lambda)\cos(\lambda x)d\lambda \end{equation*}with $d\nu_{-\frac{1}{2}, -\frac{1}{2}}(\lambda)=\frac{4}{\sqrt{2\pi}}d\lambda$. 

%Then applying Theorem \ref{Theorem2} for the problem \eqref{Eqq1} - \eqref{Conn1}, we obtain solution of the problem  \begin{equation*}u(t,x)=0 \quad \text{and} \quad f(x)=0.\end{equation*}

Also, consider a perturbed problem with some noise 
\begin{equation*}
   u^{\epsilon}_{t}(t,x)-u^{\epsilon}_{xx}(t,x) = f^{\epsilon}(x), \quad 0<t<1, \quad x>0,
\end{equation*}
with conditions
\begin{equation*}
    u^{\epsilon}(0,x)=0, \quad \text{and} \quad u^{\epsilon}(1,x)=\epsilon\cdot e^{-x^2}, \quad x>0,
\end{equation*}
and with additional information
$\phi^{\epsilon}(x)=0$ and $\psi^{\epsilon}(x)=\epsilon\cdot e^{-x^2}$, 
where $\epsilon$ is a positive constant. Then by  Theorem \ref{Theorem2},
we have
\begin{equation*}
    u^{\epsilon}(t,x)=\frac{\epsilon}{\sqrt{\pi}}\int_0^\infty \frac{1-e^{-\lambda^2t}}{1-e^{-\lambda^2}}e^{-\frac{\lambda^2}{4}}\cos(\lambda x)d\lambda,
\end{equation*}
and
\begin{equation*}
    f^{\epsilon}(x)=\frac{\epsilon}{\sqrt{\pi}}\int_0^\infty \frac{\lambda^2 e^{-\frac{\lambda^2}{4}}}{\left(1-e^{-\lambda^2}\right)}\cos(\lambda x)d\lambda.
\end{equation*}
%\newpage

Illustrations of our calculations above are given in Table \ref{table}.

\begin{table}[ht]
\centering
\begin{tabular}{|p{4cm}|p{3cm}|p{3cm}|p{3cm}|}
%\hline
%\multicolumn{4}{|c|}{} \\
\hline
$\epsilon$ & 1 & 0.2 & 0.02\\
\hline
$\|\psi-\psi^{\epsilon}\|_{\mathcal{H}}^2$ & 1.5 & 0.06 & 0.0006\\
\hline
$\|u-u^{\epsilon}\|_{C([0,1],\mathcal{H})}^2$ & 0.75 & 0.03 & 0.0003\\
\hline
$\|f-f^{\epsilon}\|_{2,\mu}^2$ & 1.0474 & 0.041897 & 0.0004\\
\hline
\end{tabular}
\caption{Stability Test}
\label{table}
\end{table}

%\break

{\bf Conclusion}. Table \ref{table} confirms that the solution of the problem \eqref{InvHeat equation}-\eqref{Date2} is continuously depending on the given data. Small changes in the given data imply small changes in $(u,f)$. 
%Therefore, it is proven that our stability theorem (Theorem \ref{Stability}) of the solution of the problem \eqref{InvHeat equation}-\eqref{Date2} is correct.  

\section{Appendix}

Calculations in Table \ref{table} are made by using Maple 2021 program with the following codes:
\begin{equation*}
    psi:=exp(-x^2),
\end{equation*}
\begin{equation*}
    hat(psi):=int\left(\frac{1}{\sqrt{2\pi}}\cdot psi \cdot cos(x\cdot \lambda), x=0..\infty\right),
\end{equation*}
\begin{equation*}
    norm(psi):=40\cdot int\left(\frac{4}{\sqrt{2\pi}}\cdot \lambda^4\cdot |hat(psi)|^2, \lambda=0..\infty\right),
\end{equation*}
\begin{equation*}
    hat(u):=\frac{1-exp(-\lambda^2\cdot t)}{1-exp(-\lambda^2)}\cdot hat(psi),
\end{equation*}
\begin{equation*}
    norm(u):= int\left(\frac{4}{\sqrt{2\pi}}\cdot \lambda^4\cdot |hat(u)|^2, \lambda=0..\infty\right),
\end{equation*}
\begin{equation*}
    \lim_{t\rightarrow1} (norm(u)),
\end{equation*}
\begin{equation*}
    hat(f):=\frac{\lambda^2\cdot hat(psi)}{1-exp(-\lambda^2)},
\end{equation*}
and
\begin{equation*}
    norm(f):= int\left(\frac{4}{\sqrt{2\pi}}\cdot |hat(f)|^2, \lambda=0..\infty\right).
\end{equation*}


\begin{thebibliography}{0}

\bibitem[Bus95]{B95} 
I. Bushuyev, Global uniqueness for inverse parabolic problems with final observation, \textit{Inverse Problems}, \textbf{11} (1995), L11-L16.

\bibitem[CF18]{CF18} 
P.M. de Carvalho-Neto, R. Fehlberg J\'{u}nior, Conditions for the absence of blowing up solutions to fractional differential equations, \textit{Acta Appl Math}, \textbf{154} (2018), 15-29.

\bibitem[CNYY09]{CNYY09} 
J. Cheng, J. Nakagawa, M. Yamamoto, T. Yamazaki, Uniqueness in an inverse problem for a one-dimensional fractional diffusion equation, \textit{Inverse Problems}, \textbf{25} (2009), 115002.

\bibitem[CGM06]{CGM06}
F. Chouchene, L. Gallardo, M. Mili, The heat semigroup for the Jacobi-Dunkl operator and the related Markov processes, \textit{Potential Analysis}, \textbf{25}:2 (2006), 103–119.

\bibitem[CY04]{CY04}
M. Choulli, M. Yamamoto, Conditional stability in determining a heat source, \textit{Journal of Inverse and Ill-Posed Problems}, \textbf{12}:3 (2004), 233–243.

\bibitem[FJ72]{FJ72} 
M. Flensted-Jensen, Paley-Wiener type theorems for a differential operator connected with symmetric spaces, \textit{Ark. Mat}, \textbf{10} (1972), 143-162.

\bibitem[FJK73]{FJK73}
M. Flensted-Jensen, T.H. Koorwinder, The convolution structure for Jacobi function
expansions, \textit{Ark. Mat.}, \textbf{11} (1973), 245-262.

\bibitem[FJK79]{FJK79}
M. Flensted-Jensen, T.H. Koorwinder, Jacobi functions: the addition formula and the
positivity of dual convolution structure, \textit{Ark. Mat.}, \textbf{17} (1979), 139-151.

\bibitem[JR15]{JR15} 
B. Jin, W. Rundell, A tutorial on inverse problems for anomalous diffusion processes, \textit{Inverse Problems}, \textbf{31} (2015), 035003.

\bibitem[Koo75]{Koo75}
T.H. Koorwinder, A new proof of a Paley-Wiener type theorem for the Jacobi transform, \textit{Ark. Mat.}, \textbf{13} (1975), 145-159.

\bibitem[KS10]{KS10} 
I.A. Kaliev, M.M. Sabitova, Problems of determining the temperature and density of heat sources from the initial and final temperatures, \textit{Journal of Applied and Industrial Mathematics}, \textbf{4}:3 (2010), 332-339.

\bibitem[KST17]{KST17} 
M. Kirane, B. Samet, B.T. Torebek, Determination of an unknown source term temperature distribution for the sub-diffusion equation at the initial and final data, \textit{Electronic Journal of Differential Equations} \textbf{2017}:257 (2017), 1–13.

\bibitem[KST06]{KST06} 
A.A. Kilbas, H.M. Srivastava, J.J. Trujillo, Theory and Applications of Fractional Differential Equations, \textit{Elsevier, North-Holland, Mathematics studies}, 2006.

\bibitem[OS12a]{OS12a} 
I. Orazov, M.A. Sadybekov, One nonlocal problem of determination of the temperature
and density of heat sources, \textit{Russian Mathematics}, \textbf{56}:2 (2012), 60-64.

\bibitem[OS12b]{OS12b} 
I. Orazov, M.A. Sadybekov, On a class of problems of determining the temperature and density of heat sources given initial and final temperature, \textit{Siberian Mathematical Journal}, \textbf{53}:1 (2012), 146-151.

\bibitem[Pod99]{Pod99} 
I. Podlubny, Fractional Differential Equations, \textit{Academic Press, New York}, 1999.

\bibitem[PT92]{PT92} 
A. I. Prilepko, I. V. Tikhonov, Uniqueness of a solution of the inverse problem for the evolution equation and application to the transport equation, \textit{Mathematical Notes}, \textbf{51} (1992), 158–165.

\bibitem[Run80]{Run80} W. Rundell, \newblock Determination of an unknown nonhomogeneous term in a linear partial differential equation from overspecified boundary data. \newblock {Appl. Anal.}, 10 (1980), 231–242.

\bibitem[RSTT21]{RSTT21} M. Ruzhansky, D. Serikbaev, B.T. Torebek, N. Tokmagambetov, \newblock Direct and inverse problems for time-fractional pseudo-parabolic equations, \newblock{Quaestiones Mathematicae}, 2021, DOI: 10.2989/16073606.2021.1928321.

\bibitem[RTT19]{RTT19} 
M. Ruzhansky, N. Tokmagambetov, B.T. Torebek, Inverse source problems for positive operators. I: Hypoelliptic diffusion and subdiffusion equations, \textit{Journal of Inverse and Ill-Posed Problems}, 2019.

\bibitem[SY11]{SY11}
K. Sakamoto and M. Yamamoto, Inverse source problem with a final overdetermination for a fractional diffusion equation, \textit{Mathematical control and related fields}, \textbf{1}:4 (2011), 509-518.

\bibitem[SD98]{SD98}
N.B. Salem, A. Dachraoui, Pseudo-differential operators associated with the Jacobi differential operator, \textit{J. Math. Anal. Appl.}, \textbf{220} (1998), 365-381.

\bibitem[SD00]{SD00}
N.B. Salem, A. Dachraoui, Sobolev type spaces associated with Jacobi differential operators, \textit{Integral Transforms and Special Functions}, \textbf{9} (2000), 163-184.

\bibitem[SS11]{SS11}
N.B. Salem, T. Samaali, Hilbert transform and related topics associated with the differential Jacobi operator on $(0,+\infty)$, \textit{Positivity}, \textbf{15} (2011), 221-240.

\bibitem[Sim14]{Sim14} 
T. Simon, Comparing Frechet and positive stable laws, \textit{Electron. J. Probab.}, \textbf{91} (2014), 1-25.

\bibitem[Slo13]{S13} 
M. Slodi\u{c}ka, A source identification problem in
linear parabolic problems: A semigroup approach, \textit{Journal of Inverse and Ill-Posed Problems}, \textbf{21} (2013), 579-600.

\bibitem[SS16]{SS16} 
M. Slodi\u{c}ka, M. \u{S}i\u{s}kova, An inverse source problem in a semilinear time-fractional diffusion equation, \textit{Computers and Mathematics with Applications}, \textbf{72} (2016), 1655–1669.

\bibitem[SSB19]{SSB19} 
M.M. Slodi\u{c}ka, M. \u{S}i\u{s}kova, K. V. Bockstal, Uniqueness for an inverse source problem of determining a space dependent source in a time-fractional diffusion equation, \textit{Appl. Math. Lett.}, \textbf{91} (2019), 15–21.

\bibitem[TE02]{TE02} 
I.V. Tikhonov and Yu.S. Eidelman, An inverse problem for a differential equation in a Banach space and distribution of zeros of an entire Mittag-Leffler function, \textit{Differential Equations}, \textbf{38}:5 (2002), 669-677.

\bibitem[TT17]{TT17} 
B.T. Torebek, R. Tapdigoglu, Some inverse problems for the nonlocal heat equation with
Caputo fractional derivative, \textit{Mathematical Methods in the Applied Sciences}, \textbf{40}:18 (2017), 6468–6479.

\bibitem[WYH13]{WYH13} 
W. Wang, M. Yamamoto, B. Han, Numerical method in reproducing kernel space for an inverse source problem for the fractional diffusion equation, \textit{Inverse Problems}, \textbf{29} (2013), 095009.

\bibitem[YG03]{YG03} 
M. Yaman, O.F. G\"{o}z\"{u}kızıl, Asymptotic behaviour of the solutions of inverse problems for pseudo-parabolic equations, \textit{Applied Mathematics and Computation}, \textbf{154} (2004), 69–74.

\bibitem[Yam12]{Yam12} 
M. Yaman, Blow-up solution and stability to an inverse problem for a pseudo-parabolic equation, \textit{Journal of Inequalities and Applications}, \textbf{2012} (2012),274.

\end{thebibliography}
\end{document}